\documentclass[11pt,a4]{article}

\usepackage{pstricks, pst-coil, pst-node, pst-tree, multido}
\usepackage{enumerate}
\usepackage{latexsym} 
\usepackage{theorem}
\usepackage{graphics}
\usepackage{graphicx}
\usepackage{amsmath}
\usepackage{amssymb}

\newtheorem{defeng}{Definition}[section]
\newtheorem{theorem}[defeng]{Theorem}

\newtheorem{lemma}[defeng]{Lemma}

{\theorembodyfont{\rmfamily} }
{\theorembodyfont{\rmfamily} }
{\theorembodyfont{\rmfamily} }
{\theoremstyle{break}\theorembodyfont{\rmfamily} }
{\theoremstyle{break}\theorembodyfont{\rmfamily} }

\newcommand{\tp}{}

\newcounter{claim}

\newenvironment{proof}[1][]%
 {\noindent {\setcounter{claim}{0}\sc proof ---
   }{#1}{}}{\hfill$\Box$\vspace{2ex}} 

{\refstepcounter{claim}\vspace{1ex}\noindent{(\it\arabic{claim}){#1}{}}\it}{\vspace{1ex}}

	{\noindent {}{#1}{}}{ This proves~(\arabic{claim}).\vspace{1ex}}

\newcommand{\sm}{\setminus} 

\usepackage{ifpdf}

\ifpdf
\fi

\begin{document}

\title{On Roussel--Rubio-type lemmas and their consequences}

\date{January 12, 2011}

\author{Nicolas Trotignon\thanks{CNRS, LIAFA, Universit\'e Paris 7 ---
    Paris Diderot. E-mail: nicolas.trotignon@liafa.jussieu.fr.
    Partially supported by the French \emph{Agence Nationale de la
      Recherche} under reference \textsc{anr 10 jcjc 0204 01}.},~
Kristina Vu\v skovi\'c\thanks{School of Computing, University of
  Leeds, Leeds LS2 9JT, UK and Faculty of Computer Science, Union
  University, Knez Mihailova 6/VI, 11000 Belgrade, Serbia. E-mail:
  k.vuskovic@leeds.ac.uk.  Partially supported by Serbian Ministry for
  Science and Technological Development Grant 144015G and EPSRC grant
  EP/H021426/1. \newline The two authors are also supported by PHC
  Pavle Savi\'c grant jointly awarded by EGIDE, an agency of the
  French Minist\`ere des Affaires \'etrang\`eres et europ\'eennes, and
  Serbian Ministry for Science and Technological Development.  }}

\maketitle

\begin{abstract}
  Roussel and Rubio proved a lemma which is essential in the proof of
  the Strong Perfect Graph Theorem.  We give a new short proof of the
  main case of this lemma.  In this note, we also give a short proof of
  Hayward's decomposition theorem for weakly chordal graphs, relying
  on a Roussel--Rubio-type lemma.  We recall how Roussel--Rubio-type
  lemmas yield very short proofs of the existence of even pairs in
  weakly chordal graphs and Meyniel graphs.
 \end{abstract}


\section{Introduction}

A \emph{hole} in a graph is an induced cycle of length at least~4.  An
\emph{antihole} is the complement of a hole.  A graph is a
\emph{Berge} graph if it contains no odd hole and no odd antihole,
where \emph{odd} refers to the length of the hole.  By $\chi(G)$ we
denote the chromatic number of $G$, and by $\omega(G)$ we denote the
maximum size of a clique in $G$.  A graph $G$ is \emph{perfect} if for
any induced subgraph $H$ of $G$, $\chi(H) = \omega(H)$.  Berge
conjectured that every Berge graph is perfect~\cite{berge:61}.  This was
known as the \emph{Strong Perfect Graph Conjecture}, was the object of
much research and was finally proved by Chudnovsky, Robertson, Seymour,
and Thomas~\cite{chudnovsky.r.s.t:spgt}.

A lemma due to Roussel and Rubio~\cite{roussel.rubio:01} is used at
many steps in~\cite{chudnovsky.r.s.t:spgt}.  In fact, the authors of
\cite{chudnovsky.r.s.t:spgt} rediscovered it (in joint work with
Thomassen) and initially named it the \emph{wonderful lemma} because of
its many applications.  The first aim of this note is to give a
previously unpublished short proof of the Roussel--Rubio Lemma, see
Section~\ref{sec:RR}.  In the same section, we recall variants of this
lemma for classes of perfect graphs.

A hole or an antihole is said to be \emph{long} if it contains at
least 5 vertices.  A graph is \emph{weakly chordal}, also called
\emph{weakly triangulated} if it contains no long hole and no long
antihole.  Weakly chordal graphs were investigated by
Chv\'atal~\cite{chvatal:starcutset} and Hayward et
al.~\cite{hayward:these,hayward:wt,hayward:disc,hayward.reed:chap.hole}.
Hayward proved a decomposition theorem for weakly chordal graphs that
implies their perfectness.  The second aim of this note is to give a
short proof of this decomposition theorem, relying on a version of the
Roussel--Rubio Lemma, see Section~\ref{sec:WT}.

An \emph{even pair} in a graph is a pair of vertices such that all
induced paths linking them have even length.  In Section~\ref{sec:ep},
we recall known proofs of the existence of even pairs in weakly
chordal graphs and also Meyniel graphs, all relying on
Roussel--Rubio-type lemmas.

We say that a graph $G$ \emph{contains} a graph $H$ if $G$ has an
induced subgraph isomorphic to $H$.  We say that $G$ is
\emph{$H$-free} if $G$ does not contain $H$.

\section{Roussel--Rubio-type lemmas}
\label{sec:RR}

A set of vertices in a graph is \emph{anticonnected} if it induces a
graph whose complement is connected.  A vertex is \emph{complete} to a
set $S$ if it is adjacent to all vertices in $S$.  The Roussel--Rubio
Lemma states that, in a sense, any anticonnected set of vertices of a
Berge graph behaves like a single vertex.  How does a vertex $v$
``behave'' in a Berge graph?  If a chordless path of odd length (at
least~3) has both ends adjacent to $v$, then $v$ must have other
neighbors in the path, for otherwise there is an odd hole.  The lemma
states roughly that an anticonnected set $T$ of vertices behaves
similarly: if a chordless path of odd length (at least~3) has both
ends complete to $T$, then at least one internal vertex of the path is
also complete to $T$.  In fact, there are two situations where this
statement fails, so the lemma is slightly more complicated.

Proofs of the Roussel--Rubio Lemma can be found in the original
paper~\cite{roussel.rubio:01}, and also in~\cite{confortiCVZ02}.  A
simpler proof due to Maffray and Trotignon can be found
in~\cite{nicolas:artemis}.  The proof given here, although previously
unpublished, is due to Kapoor, Vu\v skovi\'c, and Zambelli.
Originally, the argument can be found in an unpublished work on
\emph{cleaning}, by Kapoor and Vu\v skovi\'c, and Zambelli noticed
that this argument yields an elegant proof of the Roussel--Rubio Lemma
(cleaning is an essential component of the recognition algorithm for
Berge graphs, see~\cite{chudnovsky.c.l.s.v:reco}).  The proof is
essentially the same as the other proofs (in particular the proof in
\cite{nicolas:artemis}), except in the case when $T$ is a stable set.
So, we present here only this case.

When $T$ is a set of vertices of a graph $G$, a set $S\subseteq
V(G)\sm T$ is \emph{$T$-complete} if each vertex of $S$ is
adjacent to each vertex of $T$.  An \emph{antipath} is the complement
of a chordless path.  If $P= xx'\dots y'y$ is a chordless path of
length at least 3 in a graph $G$, we say that a pair of non-adjacent
vertices $\{u,v\}$, disjoint from $P$, is a \emph{leap} for $P$ if
$N(u)\cap V(P) = \{x,x',y\} $ and $N(v)\cap V(P) = \{x,y',y\}$.  Note
that the following statement is a reformulation
from~\cite{chudnovsky.r.s.t:spgt} of the original lemma.

\begin{lemma}[Roussel and Rubio~\cite{roussel.rubio:01}] 
  Let $G$ be an odd-hole-free graph and let $T$ be an anticonnected
  set of $V(G)$.  Let $P$ be a chordless path in $G\sm T$, of odd
  length, at least~3.  If the ends of $P$ are $T$-complete, then at
  least one internal vertex of $P$ is $T$-complete, or $T$ contains a
  leap for $P$, or $P$ has length 3, say $P = x x' y' y$, and $G$
  contains an antipath of length at least 3, from $x'$ to $y'$ whose
  interior is in $T$.
\end{lemma}

\begin{proof}
  We prove that if $T$ is a stable set, then at least one internal
  vertex of $P$ is $T$-complete or $T$ contains a leap for $P$ (for
  the case when $T$ is not a stable set, see the proof in
  \cite{nicolas:artemis}).  In fact, we prove slightly more by
  induction on $|T|$: either there is a leap, or $P$ has an odd number
  of $T$-complete edges.

  Let $P= xx'\dots y'y$.  Mark the vertices of $P$ that have at least
  one neighbor in $T$.  Call an \emph{interval} any subpath of $P$, of
  length at least~1, whose ends are marked and whose internal vertices
  are not.  Since $x$ and $y$ are marked, the edges of $P$ are
  partitioned by the intervals of $P$.

  If $|T|=1$, then every interval $P'$ of odd length has length~1, for
  otherwise $T \cup P'$ induces an odd hole.  Since the length of $P$
  is odd, it has an odd number of intervals of odd length, and hence
  an odd number of $T$-complete edges.  Now suppose $|T|>1$ and there
  is no leap for $P$ contained in $T$.
 
  We claim that every interval of $P$ either has even length or has
  length~1.  Indeed, suppose there is an interval of odd length, at
  least 3, say $P' = x''\dots y''$, named so that $x, x'', y'', y$
  appear in this order along $P$.  Let $u$ and $v$ be neighbors of $x''$
  and $y''$ in $T$, respectively.  If $x''$ and $y''$ have a common
  neighbor $t$ in $T$ then $P'\cup \{t\}$ induces an odd hole.  Hence
  $u\neq v$, $x''\neq x$, $y''\neq y$, $v$ is not adjacent to $x''$,
  and $u$ is not adjacent to $y''$.  If $x''\neq x'$, then $P'\cup \{
  u, x, v\}$ induces an odd hole. So, $x''=x'$ and similarly,
  $y''=y'$.  Hence, $\{u, v\}$ is a leap, a contradiction.  This
  proves our claim.

  Hence, there is an odd number of intervals of length $1$ in $P$.
  Moreover, we claim that for every interval of length~1, there is a
  vertex in $T$ adjacent to both its ends.  Indeed, suppose that there
  is an interval $x''y''$ such that $x''$ and $y''$ do not have a
  common neighbor in $T$.  Let $u$ be a neighbor of $x''$ in $T$, and
  let $v$ be a neighbor of $y''$ with $u \neq v$, $uy''\not\in E(G)$,
  and $vx'' \not\in E(G)$.  Note that $x\neq x''$ and $y\neq y''$.  If
  $x''\neq x'$, then $\{ u, x, v, x'', y''\}$ induces an odd hole.
  So, $x''=x'$ and similarly $y''=y'$.  Now $\{u, v\}$ is a leap, a
  contradiction.

  For every $v\in T$, denote by $f(v)$ the set of all $\{v\}$-complete
  edges of $P$.  Let $v_1, \dots, v_n$ be the elements of $T$. We know
  that $|f(v_1) \cup \dots \cup f(v_n)|$ is odd, since, from the
  previous paragraph, it is equal to the number of the intervals of
  length~1. Moreover, by the sieve formula, also known as the
  inclusion-exclusion formula, we have:

  \begin{eqnarray}
    |f(v_1)   \cup \dots \cup f(v_n)| & = & \sum_{i} |f(v_i)| \nonumber \\
    & & -  \sum_{i\neq j} |f(v_i) \cap f(v_j)| \nonumber \\
    & & \vdots \nonumber \\
    & & + (-1)^{(k+1)} \sum_{I\subset \{1, \dots, n\}, |I|=k} |\cap_{i\in I} f(v_i)| \nonumber \\
    & & \vdots \nonumber \\
    & & + (-1)^{(n+1)} |f(v_1) \cap \dots \cap f(v_n)| \nonumber
  \end{eqnarray}

  By the induction hypothesis, we know that if $S\subsetneq T$, then
  $P$ has an odd number of $S$-complete edges (note that a leap in $S$
  is a leap in $T$, and we are assuming that $T$ has no leap).  Hence if
  $I\subsetneq \{1, \dots, n\}$, then $|\cap_{i\in I} f(v_i)|$ is odd.
  Thus, we can rewrite the above equality modulo 2 as:
  
  \[
  |f(v_1) \cup  \dots \cup f(v_n)|= (2^n-2)  +
  (-1)^{(n+1)} |f(v_1)  \cap \dots \cap f(v_n)|
  \]
    
  Since $|f(v_1)  \cup \dots \cup f(v_n)|$ is odd, it
  follows that $|f(v_1)  \cap \dots \cap f(v_n)|$ is odd,
  meaning that $P$ has an odd number of $T$-complete edges.
\end{proof}

We now present several Roussel--Rubio-type lemmas.  It seems that the
first Roussel--Rubio-type lemma ever proved is the following simple
lemma about Meyniel graphs, originally pointed out by Meyniel (and
needed in Section~\ref{sec:ep}).  A graph is a \emph{Meyniel} graph if
every odd cycle of length at least~5 has at least two chords.  We
include the original proof for completeness.

\begin{lemma}[Meyniel~\cite{meyniel:87}]
  \label{l:Mwt}
  Let $G$ be a Meyniel graph, and $v\in V(G)$.  Let $P$ be
  a path of $G\sm v$ of odd length, at least 3.  If the  ends of $P$ are
  adjacent to $v$,  then all vertices of $P$ are adjacent to $v$.
\end{lemma}

\begin{proof}
  Let an \emph{interval} be any subpath of $P$, of length at least 1,
  whose ends are adjacent to $v$ and whose internal vertices are not.
  The length of any interval is~1 or is even, because otherwise,
  together with $v$, it induces an odd hole.  So, if $v$ is not
  complete to $P$, then $P$ does contain intervals of even length and
  also intervals of lengths~1 since $P$ has odd length.  So, there is
  an interval of length~1 and an interval of even length that are
  consecutive.  This yields an odd cycle with a unique chord, a
  contradiction.
\end{proof}

Another Roussel--Rubio-type lemma is used in \cite{hayward:disc} to
prove that every vertex in a minimal imperfect graph is in a long hole
or in a long antihole, see Lemma 2 in~\cite{hayward:disc}.  The
following can be seen as a Roussel--Rubio-type lemma for weakly
chordal graphs; that is used in Sections~\ref{sec:WT}
and~\ref{sec:ep}.  We include the proof from~\cite{nicolas:artemis}
for completeness.

\begin{lemma}[Maffray and Trotignon~\cite{nicolas:artemis}]
  \label{l:rrwt}
  Let $G$ be a weakly chordal graph, and let $T \subseteq V(G)$ be an
  anticonnected set.  Let $P = x \dots y $ be a chordless path of $G\sm T$ of
  length at least 3.  If  the ends of $P$ are $T$-complete,  then $P$ has an
  internal vertex that is $T$-complete.
\end{lemma}

\begin{proof}
  Note that no vertex $t \in T$ can be non-adjacent to two consecutive
  vertices of $P$, for otherwise $V(P) \cup \{t\}$ contains a long
  hole.  Let $z$ be an internal vertex of $P$ adjacent to a maximum
  number of vertices of $T$.  Suppose for a contradiction that there
  exists a vertex $u \in T \setminus N(z)$.  Let $x'$ and $y'$ be the
  neighbors of $z$ along $P$, so that $x, x', z, y', y$ appear in this
  order along $P$.  Then, from the first sentence of this proof, and
  subject to the conditions established so far, $ux', uy' \in E(G)$.
  Without loss of generality, we may assume $x' \neq x$.  From the
  choice of $z$, since $ux' \in E(G)$ and $uz \notin E(G)$, there
  exists a vertex $v \in T$ such that $vz\in E(G)$ and $vx'\notin
  E(G)$.  Since $G[T]$ is anticonnected, there exists an antipath $Q$
  of $G[T]$ from $u$ to $v$.  Suppose that $u, v$ are chosen subject
  to the minimality of this antipath.  From the first sentence and
  subject to the conditions established so far,
  internal vertices of $Q$ are all adjacent to $x'$ or $z$ and from
  the minimality of $Q$, internal vertices of $Q$ are all adjacent to
  $x'$ and $z$.  If $x'x \notin E(G)$ then $V(Q) \cup \{ x, x', z\}$
  induces a long antihole. So $x'x \in E(G)$.  If $zy \notin E(G)$
  then $V(Q) \cup \{z, x', y\}$ induces a long antihole.  So $zy \in
  E(G)$ and $y = y'$.  Now, $V(Q) \cup \{x, x', z, y\}$ induces a
  long antihole, a contradiction.
\end{proof}

Let $C(T)$  denote the set of all $T$-complete vertices.  The
following is implicit in~\cite{nicolas:artemis}.

\begin{lemma}
  \label{l:pathWT}
  Let $G$ be a weakly chordal graph, and $T\subseteq V(G)$ a set of
  vertices such that $G[T]$ is anticonnected and $C(T)$ contains at
  least two non-adjacent vertices.  If $T$ is inclusion-wise maximal
  with respect to these properties, then any chordless path of $G\sm
  T$ whose ends are in $C(T)$ has all its vertices in $C(T)$.
\end{lemma}

\begin{proof}
  Let $P$ be a chordless path in $G\sm T$ whose ends are in $C(T)$.  If some
  vertex of $P$ is not in $C(T)$, then $P$ contains a subpath $P'$ of
  length at least 2 whose ends are in $C(T)$ and whose interior is
  disjoint from $C(T)$.  If $P'$ is of length~2, say $P'=a\tp t \tp b$,
  then $T\cup \{t\}$ is a set that contradicts the maximality of $T$.
  If $P'$ is of length greater than~2, then it contradicts
  Lemma~\ref{l:rrwt}.
\end{proof}

\section{Weakly chordal graphs}
\label{sec:WT}

Here we give a new simple proof of Hayward's decomposition theorem for
weakly chordal graphs.  A \emph{cutset}\index{cutset} in a graph is a
set $S$ of vertices such that $G\sm S$ is disconnected.  A \emph{star
  cutset} of a graph is a set of vertices $S$ that contains a vertex
$c$ such that $S \subseteq \{c\}\cup N(c))$ and that is a cutset.  

\begin{theorem}[Hayward \cite{hayward:these,hayward.reed:chap.hole}]
  \label{th:wt}
  If  $G$ is a weakly chordal graph,  then one of the following holds: 
  \begin{itemize}
  \item $G$ is a complete graph;
  \item $G$ is the complement of a perfect matching;
  \item $G$ admits a star cutset.
  \end{itemize}
\end{theorem}

\begin{proof}
  We proceed by induction on $|V(G)|$.  If $G$ is a disjoint union of
  complete graphs (in particular when $|V(G)| = 1$), then the theorem
  holds: if there is more than one component, then some vertex is a
  star cutset, unless $G$ has exactly 2 vertices, that are furthermore
  not adjacent, in which case it is the complement of a perfect
  matching.  Otherwise, we may assume that $G$ contains a chordless
  path $P$ on~3 vertices.  Hence, there exists a set $T$ of vertices
  such that $G[T]$ is anticonnected and $C(T)$ contains at least two
  non-adjacent vertices, because the center of $P$ forms such a set
  $T$.  Let us assume that $T$ is maximal, as in Lemma~\ref{l:pathWT}.
  Since $C(T)$ is not a clique, by the induction hypothesis, we have
  two cases to consider:  the graph induced by $C(T)$ has a star cutset
  $S$, or the graph induced by $C(T)$ is the complement of a perfect
  matching.  In the first case, by Lemma~\ref{l:pathWT}, $T\cup S$ is
  a star cutset of $G$.  So, we may assume that we are in the second
  case.

  Suppose first that $V(G) = T \cup C(T)$.  Then by the induction
  hypothesis, either $T$ is a clique (in which case $|T|=1$ since $T$
  is anticonnected), or $T$ induces the complement of a perfect
  matching or $T$ has a star cutset $S$.  If $|T|=1$, then
  $T \cup C(T) \sm \{x, y\}$, where $x, y$ are two non-adjacent
  vertices in $C(T)$, is a star cutset of $G$.  In the second case,
  $G$ itself is the complement of a perfect matching. In the third
  case, $S\cup C(T)$ is a star cutset of $G$.

  So, we may assume that there exists a vertex $x \in V(G) \sm (T \cup
  C(T))$.  We choose $x$ with a neighbor $y$ in $C(T)$; this is
  possible otherwise $T$ together with any vertex of $C(T)$ forms a
  star cutset of $G$.  

  Recall that $C(T)$ is the complement of a perfect matching.  Let
  $y'$ be the non-neighbor of $y$ in $C(T)$.  We claim that $S = T
  \cup C(T) \sm \{y'\}$ is a star cutset of $G$ separating $x$ from
  $y'$.  First, observe that $S\subseteq \{y\}\cup N(y)$.  Also, if
  there is a path in $G\sm S$ from $x$ to $y'$, then there is a
  chordless path, and by appending $y$ to that chordless path we see
  that $G\sm T$ contains a chordless path from $y$ to $y'$ that is not
  included in $C(T)$, which contradicts Lemma~\ref{l:pathWT}.  This
  proves our claim.
\end{proof}

An easy corollary of Theorem~\ref{th:wt} is another theorem of
Hayward~\cite{hayward:wt} stating that if $G$ is weakly chordal
on at least 3 vertices then  $G$ or $\overline{G}$ admits a star
cutset.  This implies the perfectness of
weakly chordal graphs because Chv\'atal~\cite{chvatal:starcutset}
proved that a minimally imperfect graph has no star cutset.

\section{Even pairs}
\label{sec:ep}

In this section, neither the results nor their proofs are new, but we
include them because we think that it is interesting to see how they
all rely on Roussel--Rubio-type lemmas.

Even pairs are a tool to prove perfectness of graphs and to give
polynomial time coloring algorithms (see~\cite{everett.f.l.m.p.r:ep}).
Roussel--Rubio-type lemmas provide  a good tool to prove the existence
of an even pair.  Even pairs are used in~\cite{nicolas:artemis} to give a
polynomial time coloring algorithm for a class of graphs that
generalizes Meyniel graphs, weakly chordal graphs, and perfectly
orderable graphs, the class of so-called \emph{Artemis graphs}.  They are used
in~\cite{chudnovsky.seymour:even} to significantly shorten the proof
of the Strong Perfect Graph Theorem.  Here we give two very simple
examples of this technique.  The first one is due to Meyniel.

\begin{theorem}[Meyniel \cite{meyniel:87}]
  If $G$ is a Meyniel graph, then either $G$ is a complete graph or $G$ has an
  even pair. 
\end{theorem}

\begin{proof}
  We proceed by induction on $|V(G)|$.  If $G$ is a disjoint union of
  complete graphs (in particular when $|V(G)| = 1$), then the theorem holds:
  if there is more than one component, then an even pair is obtained
  by taking two vertices in different components.  Otherwise, $G$ contains a
  chordless path $P$ on 3 vertices.  Let $v$ be the internal vertex of $P$.  By
  the induction hypothesis, $N(v)$ contains an even pair for $G[N(v)]$.  By
  Lemma~\ref{l:Mwt}, this is an even pair for $G$.
\end{proof}

Since Meyniel~\cite{meyniel:87} proved that a minimally imperfect
graph does not contain an even pair, the theorem above implies that
Meyniel graphs are perfect.  The following is originally due to
Hayward, Ho\`ang, and Maffray~\cite{hayward.hoang.m:90}, but the proof
given here is implicitly given in~\cite{nicolas:artemis}.  A
\emph{2-pair} of vertices is a pair $a, b$ such that all chordless
paths linking $a$ to $b$ have length~2.

\begin{theorem}[Hayward, Ho\`ang, and Maffray~\cite{hayward.hoang.m:90}]
  \label{th:2pair}
  If $G$ is a weakly chordal graph then either $G$ is a clique or
  $G$ has 2-pair.
\end{theorem}

\begin{proof}
  We proceed by induction on $|V(G)|$.  If $G$ is a disjoint union of
  complete graphs (in particular when $|V(G)| = 1$), then the theorem
  holds trivially.  We may therefore assume that $G$ contains a
  chordless path $P$ on 3 vertices.  Hence there exists a set $T$ as
  in Lemma~\ref{l:pathWT} (start with the center of $P$ to build $T$).
  Since $C(T)$ is not a clique, by the induction hypothesis, we know
  that $C(T)$ admits a 2-pair of $G[C(T)]$.  By Lemma~\ref{l:pathWT},
  it is a 2-pair of $G$.
\end{proof}

\end{document}